\documentclass{amsart}
\usepackage{graphics}
\usepackage{graphicx}
\usepackage{amsfonts,amsmath,mathrsfs}
\begin{document}

 \newtheorem{theorem}{Theorem}[section]
 \newtheorem{coro}[theorem]{Corollary}
 \newtheorem{lemma}[theorem]{Lemma}{\rm}
 \newtheorem{proposition}[theorem]{Proposition}

 \newtheorem{defn}[theorem]{Definition}{\rm}
 \newtheorem{ass}[theorem]{Assumption}
 \newtheorem{remark}[theorem]{Remark}
 \newtheorem{ex}{Example}
 
\numberwithin{equation}{section}
\newcommand{\bbR}{\mathbb{R}}
\newcommand{\bbC}{\mathbb{C}}
\newcommand{\bbZ}{\mathbb{Z}}

\def\la{\langle}
\def\ra{\rangle}
\def\glexe{\leq_{gl}\,}
\def\glex{<_{gl}\,}
\def\e{{\rm e}}

\def\fac{{\rm !}}
\def\x{\mathbf{x}}
\def\P{\mathbf{P}}
\def\S{\mathbf{S}}
\def\h{\mathbf{h}}
\def\m{\mathbf{m}}
\def\y{\mathbf{y}}
\def\bz{\mathbf{z}}
\def\F{\mathcal{F}}
\def\R{\mathbb{R}}
\def\T{\mathbf{T}}
\def\N{\mathbb{N}}
\def\D{\mathbf{D}}
\def\V{\mathbf{V}}
\def\U{\mathbf{U}}
\def\K{\mathbf{K}}
\def\Q{\mathbf{Q}}
\def\M{\mathbf{M}}
\def\oM{\overline{\mathbf{M}}}
\def\O{\mathbf{O}}
\def\C{\mathbb{C}}
\def\Z{\mathbb{Z}}
\def\bZ{\mathbf{Z}}
\def\H{\mathbf{H}}
\def\A{\mathbf{A}}
\def\V{\mathbf{V}}
\def\AA{\overline{\mathbf{A}}}
\def\B{\mathbf{B}}
\def\c{\mathbf{C}}
\def\L{\mathscr{L}}
\def\bS{\mathbf{S}}
\def\I{\mathbf{I}}
\def\Y{\mathbf{Y}}
\def\X{\mathbf{X}}
\def\cX{\mathbf{X}}
\def\G{\mathbf{G}}
\def\f{\mathbf{f}}
\def\z{\mathbf{z}}
\def\v{\mathbf{v}}
\def\y{\mathbf{y}}
\def\x{\mathbf{x}}
\def\bI{\mathbf{I}}
\def\y{\mathbf{y}}
\def\g{\mathbf{g}}
\def\w{\mathbf{w}}
\def\b{\mathbf{b}}
\def\a{\mathbf{a}}
\def\p{\mathbf{p}}
\def\u{\mathbf{u}}
\def\bv{\mathbf{v}}
\def\q{\mathbf{q}}
\def\e{\mathbf{e}}
\def\s{\mathcal{S}}
\def\cc{\mathcal{C}}
\def\co{{\rm co}\,}
\def\tg{\tilde{g}}
\def\tx{\tilde{\x}}
\def\tg{\tilde{g}}
\def\tA{\tilde{\A}}
\def\bell{\boldsymbol{\ell}}
\def\bxi{\boldsymbol{\xi}}
\def\balpha{\boldsymbol{\alpha}}
\def\bbeta{\boldsymbol{\beta}}
\def\bgamma{\boldsymbol{\gamma}}
\def\bpsi{\boldsymbol{\psi}}
\def\bmu{\boldsymbol{\mu}}
\def\supmu{{\rm supp}\,\mu}
\def\supp{{\rm supp}\,}
\def\cd{\mathcal{C}_d}
\def\cok{\mathcal{C}_{\K}}
\def\cop{COP}
\def\vol{{\rm vol}\,}
\def\om{\mathbf{\Omega}}
\def\f{\mathscr{F}}
\def\la{\langle\langle}
\def\ra{\rangle\rangle}
\def\blambda{\boldsymbol{\lambda}}
\def\btheta{\boldsymbol{\theta}}
\def\bphi{\boldsymbol{\phi}}
\def\bpsi{\boldsymbol{\psi}}
\def\bnu{\boldsymbol{\nu}}
\def\bom{\boldsymbol{S}}
\def\fac{{\rm !}}
\def\tM{\hat{\M}}
\def\tv{\hat{\v}}

\title[Polynomial Optimization and Christoffel function]{Polynomial Optimization, Certificates of Positivity, and Christoffel Function}

\thanks{Research supported by the AI Interdisciplinary Institute ANITI  funding through the french program
``Investing for the Future PI3A" under the grant agreement number ANR-19-PI3A-0004. It has also received funding from the European Union's Horizon 2020 research and innovation programme under the Marie Sklodowska-Curie grant
agreement $N^o$ 813211 (POEMA).
This research is also part of the programme DesCartes and is supported by the National Research Foundation, Prime Minister's Office, Singapore under its Campus for Research Excellence and Technological Enterprise (CREATE) programme.}

\author{Jean B. Lasserre}
\address{LAAS-CNRS and Institute of Mathematics\\
University of Toulouse\\
LAAS, 7 avenue du Colonel Roche\\
31077 Toulouse C\'edex 4, France\\
Tel: +33561336415}
\email{lasserre@laas.fr}

\date{}

\begin{abstract}We briefly recall basics of the Moment-SOS hierarchy in polynomial optimization and the Christoffel-Darboux kernel (and the Christoffel function (CF)) in theory of approximation and orthogonal polynomials. We then (i) show  a strong link between the CF and the SOS-based positive certificate at the core of the Moment-SOS hierarchy, and (ii) describe how the CD-kernel provides a simple interpretation of the 
SOS-hierarchy of lower bounds as searching for some signed polynomial density (while the SOS-hierarchy of upper bounds 
is searching for a positive (SOS) density). This link between the CF and positive certificates, in turn allows us (i) to establish a disintegration property of the CF
much like for measures, and (ii) for certain sets, to relate the CF of their equilibrium measure with a certificate of positivity on the set, for constant polynomials.
\end{abstract}
\maketitle
\tableofcontents

\section{Introduction}

In this chapter we describe (in our opinion, surprising) links between different fields, 
namely optimization -- convex duality -- certificates of positivity in real algebraic geometry on the one hand, and orthogonal polynomials -- Christoffel function -- approximation -- equilibrium measures, on the other hand.
More precisely, consider the polynomial optimization problem:
\begin{equation}
\label{def-pb-P}
\P:\quad f^*\,=\,\min\,\{f(\x):\: \x\,\in\,\bom\,\}\,,
\end{equation}
where $f$ is a polynomial and $\bom$ is a basic semi-algebraic set\footnote{A basic semi-algebraic set
in the intersection of finitely many sublevel sets of polynomials.}. Importantly, $f^*$ in \eqref{def-pb-P} is understood as the \emph{global} minimum of $\P$ and \emph{not} a local minimum.
As a polynomial optimization problem, $\P$ is NP-hard in general. However, in the early 2000
the Moment-SOS hierarchy (SOS stands for ``sum-of-squares") has emerged as a new methodology for solving $\P$. Its distinguishing feature is (i) to exploit powerful certificates of positivity from real algebraic geometry (and dual results on the $\bom$-moment problem) and (ii) combine them with the computational power of semidefinite programming in conic optimization, to obtain a \emph{hierarchy} of (convex) semidefinite relaxations of $\P$ of increasing size. 

The optimal values of such semidefinite relaxations provide a \emph{monotone non decreasing sequence} of certified lower bounds which converges to the global minimum $f^*$.
In addition, finite convergence is generic and when there are finitely many global minimizers, they can be obtained (also generically) from the optimal solutions of the exact semidefinite relaxation, via a simple linear algebra routine. 

Moreover, this methodology is easily adapted to solve the \emph{Generalized Moment Problem} (GMP) whose list of potential applications in mathematics, computer science, probability \& statistics, quantum information, and many areas of engineering, is almost endless. For a
detailed description of the methodology and an account of several of its applications, the interested reader is referred to e.g. the books \cite{book-2,book-1,CUP} and the many references therein. 
Less known is another (still SOS-based) hierarchy but now with an associated monotone non increasing sequence 
of \emph{upper bounds} which converges to $f^*$. While very general in its underlying principle, its practical implementation requires the feasible set $\bom$ to have a ``simple" geometry like a box, a simplex, an ellipsoid, a hypercube, or their image by an affine mapping, and recently, rates of its asymptotic convergence have been obtained in e.g. \cite{deKlerk-1,slot-1,slot-2,slot-3}.

Crucial at each step $t$ of the Moment-SOS hierarchy of lower bounds, is a dual pair 
of semidefinite programs associated with a dual pair $(C_t,C_t^*)$ of convex cones.
By a duality result of Nesterov \cite{nesterov},  the respective interiors of $C_t$ and $C^*_t$
are in a simple one-to-one correspondence. In fact, its recent interpretation in \cite[Lemma 3]{cras-1} states that every polynomial 
$p\in \mathrm{int}(C_t)$ has a distinguished 
SOS-based representation in terms of Christoffel functions associated with some moment-sequence 
$\bphi_p\in\mathrm{int}(C^*_t)$. (In particular, every degree-$2t$ SOS $p$ in the interior 
of the convex cone $\Sigma_t$ of SOS of degree at most $2t$, is the reciprocal of the Christoffel function of some linear functional $\bphi_p\in\Sigma_t^*$). In turn this duality result can be exploited to reveal additional properties of the CF. For instance we use it to
obtain a \emph{disintegration} property of the CF \cite{cras-1}, very much in like for measures on a Cartesian product of Borel spaces. 
Also, for certain compact sets
we can relate the CF of their equilibrium measure with a certain SOS-based representation of constant polynomials. Finally,
we reveal an interpretation of the latter representation \cite{cras-3} related to what we call a \emph{generalized polynomial Pell equation}
(an equation which originates in algebraic number theory).

So in this chapter we first briefly review basics of the moment-SOS hierarchies of lower and upper bounds.
We next introduce the Christoffel-Darboux kernel (CD-kernel) and the Christoffel function (CF) 
and describe some of their basic properties, which in our opinion are interesting on their own and should deserve more
attention from the optimization community. We then describe 
our interpretation of Nesterov's duality result to establish a strong link between the Christoffel functions and
the SOS-based positivity certificate used in the Moment-SOS hierarchy.
Conversely, we also describe how this duality result of convex analysis can be used to provide 
a disintegration property of the Christoffel function 
and a result on equilibrium measures of certain compact semi-algebraic sets. 

We hope that this brief account on links between seemingly distinct disciplines will raise curiosity from the optimization community.

\section{Notation, definitions and preliminary results}
Let $\R[\x]$ denote the ring of real polynomials in the variables $\x=(x_1,\ldots,x_n)$ and let
$\R[\x]_t\subset\R[\x]$ (resp. $\Sigma[\x]_t\subset\R[\x]$) be its subset 
of polynomials of degree at most $t$ (resp. sum-of-squares (SOS) polynomials of degree at most $2t$).
Let $\N^n_t:=\{\balpha\in\N^n:\vert\balpha\vert\leq t\}$
(where $\vert\balpha\vert=\sum_i\alpha_i$) with cardinal $s(t)={n+t\choose n}$. Let $\bv_t(\x)=(\x^{\balpha})_{\balpha\in\N^n_t}$ 
be the vector of monomials up to degree $t$. Then $p\in\R[\x]_t$ reads
\[\x\mapsto p(\x)\,=\,\langle\p,\v_t(\x)\rangle\,,\quad\forall \x\in\R^n\,,\]
where $\p\in\R^{s(t)}$ is the vector of coefficients of $p$ in the basis $(\x^{\balpha})_{\balpha\in\N^n}$.

Given a closed set $\mathcal{X}\subset\R^n$, denote by $\mathscr{M}(\mathcal{X})$ 
(resp. $\mathscr{M}(\mathcal{X})_+$) the space of finite signed Borel measures 
(resp. the convex cone of finite Borel measures) on $\mathcal{X}$.
The support $\mathrm{supp}(\mu)$ of a Borel measure $\mu$ on $\R^n$ is the smallest closed set $A$ such that
$\mu(\R^n\setminus A)=0$, and such a  set $A$ is unique.\\

\noindent
{\bf Riesz linear functional}\hspace{0.2cm}
With any real sequence $\bphi=(\phi_{\balpha})_{\balpha\in\N^n}$ (in bold letter)  is associated the Riesz linear functional $\phi\in\R[\x]^*$ (not in bold) defined by:
\[p\:(=\sum_{\balpha\in\N^n}p_{\balpha}\,\x^{\balpha})\quad\mapsto\quad \phi(p)\,:=\,\sum_{\balpha\in\N^n}p_{\balpha}\,\phi_{\balpha}\,=\,\langle\p,\bphi\rangle\,,\quad\forall p\in\R[\x]\,.\]
A sequence $\bphi$ has a representing measure 
if there exists a Borel measure $\phi\in\mathscr{M}(\R^n)_+$ such that $\phi_{\balpha}=\int \x^{\balpha}\,d\phi$ for all $\balpha\in\N^n$, in which case 
\[\phi(p)\,=\,\int p\,d\phi\,,\quad \forall p\in\R[\x]\,.\]
Given a sequence $\bphi=(\phi_{\balpha})_{\balpha\in\N^n}$ and a polynomial $g\in\R[\x]$ ($\x\mapsto g(\x):=\sum_{\bgamma}g_{\bgamma}\,\x^{\bgamma}$), denote by $g\cdot\bphi$ the new sequence
$(g\cdot\bphi)_{\balpha}:=\sum_{\bgamma}g_{\bgamma}\,\phi_{\balpha+\bgamma}$, $\balpha\in\N^n$,
with associated Riesz linear functional $g\cdot\phi\in\R[\x]^*$:
\[g\cdot\phi(p)\,=\,\phi(g\,p)\,,\quad\forall p\in\R[\x]\,.\]
\noindent
{\bf Moment matrix}\hspace{0.2cm}
With $t\in\N$, the moment matrix $\M_t(\phi)$ associated with a real sequence 
$\bphi=(\phi_{\balpha})_{\balpha\in\N^n}$
 is the real symmetric matrix $\M_t(\bphi)$
with rows and columns indexed by $\N^n_t$, and with entries
\[\M_t(\bphi)(\balpha,\bbeta)\,:=\,\phi(\x^{\balpha+\bbeta})\,=\,\phi_{\balpha+\bbeta}\,,\quad\balpha,\bbeta\in\N^n_t\,.\]
If $\bphi$ has a representing measure $\phi$ then necessarily 
$\M_t(\bphi)$ is positive semidefinite 
(denoted $\M_t(\bphi)\succeq0$ or $\M_t(\phi)\succeq0$) for all $t$. But the converse is not true in general.

\noindent
{\bf Localizing matrix}\hspace{0.2cm}
Similarly, with $t\in\N$, the localizing matrix $\M_t(g\cdot\bphi)$ associated with a real sequence 
$\bphi=(\phi_{\balpha})_{\balpha\in\N^n}$ and a polynomial $\x\mapsto g(\x)=\sum_{\bgamma}g_{\bgamma}\x^{\bgamma}$, is the real symmetric matrix $\M_t(g\cdot\bphi)$
with rows and columns indexed by $\N^n_t$, and with entries
\[\M_t(g\cdot\bphi)(\balpha,\bbeta)\,:=\,g\cdot\phi(\x^{\balpha+\bbeta})\,=\,\phi(g\,\x^{\balpha+\bbeta})\,=\,
\sum_{\bgamma}g_{\bgamma}\,\phi_{\balpha+\bbeta+\bgamma}\,,\quad\balpha,\bbeta\in\N^n_t\,.\]
Equivalently, $\M_t(g\cdot\bphi)$ is the moment matrix of the sequence $g\cdot\bphi$.\\

\noindent
{\bf Orthonormal polynomials.}
Let $\bphi=(\phi_{\balpha})_{\balpha\in\N^n}$ be a real sequence such that
$\M_t(\bphi)$ is positive definite (denoted $\M_t(\bphi)\succ0$)
for all $t$. Then with $\bphi$ one may associate a family 
of  \emph{orthonormal polynomials} 
$(P_{\balpha})_{\balpha\in\N^n}\subset\R[\x]$, i.e., which satisfy:
\begin{equation}
\label{ortho}
\phi(P_{\balpha}\cdot P_{\bbeta})\,=\,\delta_{\balpha=\bbeta}\,,\quad \forall \balpha,\bbeta\in\N^n\,,
\end{equation}
where $\delta_{\bullet}$ is the Kronecker symbol. One way to obtain the $P_{\balpha}$'s is via certain determinants 
formed from entries of $\M_t(\bphi)$, For instance, in dimension $n=1$, $P_0=1$ and 
\[P_1\,=\,\tau_1\cdot\mathrm{det}\left[\begin{array}{cc}
\phi_0 &\phi_1\\1 & x\end{array}\right]\,\quad
P_2\,=\,\tau_2\cdot\mathrm{det}\left[\begin{array}{ccc}
\phi_0 &\phi_1&\phi_2\\
\phi_1 &\phi_2&\phi_3\\
1 & x& x^2\end{array}\right]\,,\quad\mbox{etc.,}\]
with $\tau_k$ being a scalar that ensures
$\phi(P_k^2)=1$, $k\in\N$. For more details the interested reader is referred to e.g. \cite{Xu,annals-prob}.\\

\noindent
{\bf Putinar's Positivstellensatz.}
Let  $g_0:=1$ and  $G:=\{g_0,g_1,\ldots,g_m\}\subset\R[\x]$
with $t_g:=\lceil \mathrm{deg}(g)/2\rceil$ for all $g\in G$. Let
\begin{equation}
\label{set-S}
\bom\,:=\,\{\,\x\in\R^n\::\: g(\x)\,\geq\,0\,,\quad \forall g\,\in\,G\,\}\,,
\end{equation}
and define the sets
\begin{eqnarray}
\label{set-Q(G)}
Q(G)&=&\{\,\sum_{g\in G}\sigma_g\, g\,;\quad \sigma_g\,\in\,\Sigma[\x]\,,\:\forall g\in G\}\\
\label{set-Q_t(G)}
Q_t(G)&=&\{\,\sum_{g\in G}\sigma_g\, g\,;\quad \mathrm{deg}(\sigma_g\,g)\,\leq 2t\,,\:\forall g\in G\}\,,
\end{eqnarray}
called respectively the \emph{quadratic module} and the \emph{$t$-truncated quadratic module} associated with $G$.

\begin{remark}
\label{rem:archimedean}
With $R>0$, let $\x\mapsto \theta(\x):=R-\Vert\x\Vert^2$. The quadratic module $Q(G)\subset\R[\x]$ is said to be 
\emph{Archimedean} if there exists $R>0$ such that $\theta\in Q(G)$, in which case it provides an algebraic certificate that the set $\bom$ in \eqref{set-S} is compact.
If one knows that $\bom\subset\{\x: \Vert\x\Vert^2\leq R\}$ for some $R$ then it is a good idea to include the additional (but redundant)
constraint $R-\Vert\x\Vert^2\geq0$ in the definition \eqref{set-S} of $\bom$, in which case the resulting associated quadratic 
module $Q(G)$ is Archimedean.
\end{remark}
\begin{theorem}[Putinar \cite{putinar}]
\label{th:putinar}
Let $\bom$ be as in \eqref{set-S} and let $Q(G)$ be Archimedean. 

(i) If $p\in\R[\x]$ is strictly positive on $\bom$ then $p\in Q(G)$.

(ii) A real sequence $\bphi=(\phi_{\balpha})_{\balpha\in\N^n}$ has a representing 
Borel measure on $\bom$ if and only if $\M_t(g\cdot\bphi)\succeq0$ for all $t\in\N$, and all $g\in G$.
\end{theorem}
Theorem \ref{th:putinar} is central to prove convergence of the Moment-SOS hierarchy of
lower bounds on $f^*$, described in Section \ref{sec:lower-bounds}. \\

\noindent
{\bf Another Positivstellensatz}
We next provide an alternative Positivstellensatz where the compact set $\bom$ is not required to be semi-algebraic. Given a real sequence 
$\bphi=(\phi_{\balpha})_{\balpha\in\N^n}$, define the convex cones
\begin{eqnarray}
\label{cone-ts}
C^{\bphi}_{t,s}&:=&\{\,g\in \R[\x]_t:\: \M_s(g\cdot\bphi)\,\succeq\,0\,\}\,,\quad t\,,s\,\in\N\,.
\end{eqnarray}
Let $t$ be fixed. Observe that for each $s$, the convex cone $C^{\bphi}_{t,s}$
 is defined in terms of the single linear matrix inequality $\M_s(g\cdot\bphi)\succeq0$,
on the coefficients $(g_{\balpha})$ of $g\in\R[\x]_t$. It defines a spectrahedron in the space $\R^{s(t)}$ of 
the coefficient vector $\g=(g_{\balpha})\in\R^{s(t)}$ of $g\in\R[\x]_t$
(recall that $s(t)={n+t\choose n}$). It is a closed convex cone.

 \begin{theorem}[\cite{siopt}]
\label{th:lasserre-pos}
Let $\bom\subset\R^n$ be a compact set and let $\phi$ be an arbitrary finite 
Borel measure on $\R^n$ whose support is $\bom$ and with
moments $\bphi=(\phi_{\balpha})_{\balpha\in\N^n}$.
Then $g\in\R[\x]$ is nonnegative on $\bom$ if and only if $\M_s(g\cdot\bphi)\succeq0$ for all $s\in\N$.
\end{theorem}
Theorem \ref{th:lasserre-pos} 
is central to prove the convergence of the 
Moment-SOS hierarchy of upper bounds on $f^*$, described in Section \ref{sec:upper-bounds}.
With $t$ fixed, $(C^{\bphi}_{t,s})_{s\in\N}$ provides a monotone non increasing sequence of 
convex cones $C^{\bphi}_{t,s}$, each being an outer approximation of the convex cone
$\mathscr{C}_t(\bom)_+$  of polynomials of degree at most $t$, nonnegative on $\bom=\mathrm{supp}(\phi)$.  

In addition,
$\bigcap_{s=0}^\infty C^{\bphi}_{t,s}\,=\,\mathscr{C}_t(\bom)_+$. 
Indeed if $g\in\,C^{\bphi}_{t,s}$ for all $s$ then by Theorem \ref{th:lasserre-pos}, 
$g\in\mathscr{C}_t(\bom)_+$. Conversely, if $g\in\mathscr{C}_t(\bom)_+$ then
$\int_{\bom}p^2\,g\,d\phi\geq0$, for all $p\in\R[\x]_s$,
that is, $\M_s(g\cdot\bphi)\succeq0$, and as $s$ was arbitrary, $g\in\,C^{\bphi}_{t,s}$ for all $s$.

Notice that Theorem \ref{th:lasserre-pos} is a Nichtnegativstellensatz and applies to
sets with are not necessarily semi-algebraic. However, if on the one hand the set $\bom$ is \emph{not} required to be semi-algebraic, on the other hand
one needs to know the moment sequence $\bphi$ to exploit numerically
the convex cone $C^{\bphi}_{t,s}$. In addition, the set $\bom$ may also be non-compact. It is then enough to take a reference measure $\phi$ on $\bom$ such that $\sup_i\int e^{\vert x_i\vert}d\phi<M$ for some $M>0$; see e.g. \cite{siopt,tams}. In particular, one may then approximate from above the convex cone
$\mathscr{C}_t(\R^n)_+$ (resp. $\mathscr{C}(\R^n_+)_+$) of polynomials nonnegative on the whole $\R^n$  (resp. $\R^n_+$). (Just take $\phi=\exp(-\Vert\x\Vert^2)d\x$ on $\R^n$ (resp. 
$\phi=\exp(-2\sum_i\,x_i)\,d\x$ on $\R^n_+$).

\section{The Moment-SOS hierarchy in polynomial optimization}
\label{mom-sos}
Consider the optimization problem $\P$ in \eqref{def-pb-P}
where $f\in\R[\x]$, $\bom$ is the basic semi-algebraic set described  in \eqref{set-S},
and $f^*$ in \eqref{def-pb-P} is the \emph{global} minimum of $\P$.

\subsection{A Moment-SOS hierarchy of lower bounds}
\label{sec:lower-bounds}
\noindent
{\bf Assumption 1:}
The set $\bom$ in \eqref{set-S} is compact and contained in the Euclidean  ball of radius $\sqrt{R}$. Therefore  
with no loss of generality we may and will assume that the 
quadratic polynomial $\x\mapsto \theta(\x):=R-\Vert\x\Vert^2$ is in $G$.
Technically this implies that the quadratic module $Q(G)$ is Archimedean; see Remark \ref{rem:archimedean}.\\\

For every $g\in\R[\x]$, let $t_g:=\lceil\mathrm{deg}(g)/2\rceil$. Define $t_0:=\max[t_f,\max_{g\in G}t_g]$, and consider the sequence
of semidefinite programs indexed by $t\in\N$:
\begin{equation}
\label{relax-primal}
  \rho_t\,=\,\displaystyle\inf_{\bphi\in\R^{s(2t)}}\:\{\,\phi(f):\: \phi(1)\,=\,1\,;\:\M_{t-t_g}(g\cdot\bphi)\,\succeq\,0\,,\:\forall g\in G\,\}\,,\quad t\geq t_0\,.
 \end{equation}
  For each $t\geq t_0$, \eqref{relax-primal} is a semidefinite program and a convex relaxation 
  of \eqref{def-pb-P} so that $\rho_t\leq f^*$ for all $t\geq t_0$. In addition, the sequence $(\rho_t)_{t\geq t_0}$ is monotone non decreasing.  The dual of \eqref{relax-primal} reads:
\begin{equation}
\label{relax-dual}
  \rho^*_t\,=\,\displaystyle\sup_{\sigma_g,\lambda}\:\{\,\lambda:\:f-\lambda\,=\,\displaystyle\sum_{g\in G}\sigma_g\,g\,;\quad \sigma_g\,\in\,\Sigma[\x]_{t-t_g}\,,\forall g\in G\,\}\,.
 \end{equation}
  By weak duality between \eqref{relax-primal} and \eqref{relax-dual}, $\rho^*_t\leq\,\rho_t$  for all $t\geq t_0$ and in fact, under Assumption 1,   
  there is no duality gap, i.e., $\rho^*_t=\rho_t$ for all $t\geq t_0$; see e.g. \cite{book-1,CUP}.
 \vspace{0.2cm}
 
\noindent
{\bf KKT-optimality conditions.} In the context of problem $\P$ in \eqref{def-pb-P} with feasible set $\bom$ as in \eqref{set-S}, 
for $\x\in \bom$, let $J(\x):=\{g\in G: g(\x)=0\}$ identify the set of constraints that are 
\emph{active} at $\x$. Let $\x^*\in\bom$, and define
\begin{equation}
\label{CQ}
\mathrm{CQ}(\x^*):\quad \mbox{the vectors $(\nabla g(\x^*))_{g\in J(\x^*)}$ are 
linearly independent.}
\end{equation}
In non linear programming, the celebrated first-order necessary Karush-Kuhn-Tucker (KKT) optimality conditions state that
if $\x^*\in\bom$ is a local minimizer for $\P$ and $\mathrm{CQ}(\x^*)$ holds, then there exists $\blambda^*=(\lambda^*_g)_{g\in G}\subset\R_+$ such that 
\[\nabla f(\x^*)-\sum_{g\in G} \lambda^*_g\,\nabla g(\x^*)\,=\,0\,;\quad \lambda^*_g\,g(\x^*)\,=\,0\,,\: \forall\,g\in G\,.\]
In addition if $\lambda^*_g>0$ whenever $g(\x^*)=0$, then \emph{strict complementarity}
is said to hold. Finally,  the second-order sufficient optimality condition holds at $\x^*$ if
\[\boldmath{\u}^T\left(\nabla^2f(\x^*)-\sum_{g\in G}\lambda^*_g\,\nabla^2 g(\x^*)\,\right)\boldmath{\u}\,>\,0\,,\quad\forall \boldmath{\u}\,(\neq0)\,\in \nabla G(\x^*)^\perp\,,\]
where $\nabla G(\x^*)^\perp:=\{\boldmath{\u}\in\R^n: \boldmath{\u}^T\nabla g(\x^*)\,=\,0\,,\:\forall g\in J(\x^*)\,\}$, and $\nabla^2 h(\x^*)$ denotes the Hessian of $h$ evaluated at $\x^*$.

 \begin{theorem}
  \label{th:mom-sos}
  Let Assumption 1 hold with $\bom$ as in \eqref{set-S}, and 
  consider
  the semidefinite program \eqref{relax-primal} and its dual \eqref{relax-dual}. 
  
  (i) $\rho^*_t=\rho_t$ for all $t\geq t_0$. Moreover
  \eqref{relax-primal} has an optimal solution $\bphi^*$ for every $t\geq t_0$, and if $\bom$ has a nonnempty interior
  then \eqref{relax-dual} also has an optimal solution $(\sigma^*_g)_{g\in G}$.
  
  (ii) As $t$ increases, $\rho_t\uparrow f^*$  and finite convergence takes place if 
  \eqref{CQ}, strict complementarity, and second-order sufficiency condition, hold at every global minimizer of $\P$ (a condition that holds true generically).
  
  (iii) Let $s:=\max_{g\in G} t_g$. If 
  $\mathrm{rank}(\M_t(\bphi^*))=\mathrm{rank}(\M_{t-s}(\bphi^*))$ for some $t$, then
  $\rho_t=f^*$ (i.e.,  finite convergence takes place) and from $\bphi^*$ one may extract $\mathrm{rank}(\M_{t}(\bphi^*))$ global minimizers
  of $\P$ via a linear algebra subroutine.
  \end{theorem}

  In Theorem \ref{th:mom-sos}(iii), the (\emph{flatness}) condition on the ranks of $\M_t(\bphi^*)$ and
  $\M_{t-s}(\bphi^*)$, also holds generically (e.g. if the second-order sufficiency condition holds at every global minimizer); see e.g. \cite{baldi-1,baldi-2,nie}.
  In the recent work \cite{baldi-2}, the authors have provided the first degree-bound on the SOS weights in Putinar's positivity certificate  $f\in Q(G)$, with a \emph{polynomial} dependence on the degree of $f$  and a constant related to how far is $f$ from having a zero in $\bom$. 
 (The previous known bound of \cite{nie-markus} has an exponential dependence.)
 
 As stated in \eqref{relax-primal}, the standard Moment-SOS hierarchy does not scale well with the dimension. This is because it involves $s(2t)$ moment variables $\phi_{\balpha}$ and semidefinite matrices of size $s(t)$. Fortunately, for large-scale polynomial optimization problems, sparsity and/or symmetries are often encountered and can be exploited to obtain alternative hierarchies with much better scaling properties.  
 The interested reader is referred to the recent book \cite{magron} and the many references therein where various such techniques are described and illustrated. Also in \cite{ctp} are described first-order methods that exploit a \emph{constant trace property} of matrices of the semidefinite program \eqref{relax-primal}; they can provide an alternative to costly interior point methods for solving  large-scale semidefinite relaxations.
 
  \subsection{A Moment-SOS hierarchy of upper bounds}
 \label{sec:upper-bounds}
  In this section we now consider a hierarchy of upper bounds on the global minimum $f^*$ of 
  $\P$ in \eqref{def-pb-P} and where $\bom\subset\R^n$ is a compact set with nonempty interior.
  Let $\mu$ be a probability measure with support $\bom$ and with associated sequence of moments 
  $\bmu=(\mu_{\balpha})_{\balpha\in\N^n}$. Consider the sequence of optimization problems indexed by $t\in\N$:
  \begin{eqnarray}
 \label{upper-primal}
 \tau_t&=&\displaystyle\min_{\sigma\in\Sigma[\x]_t} 
 \{\,\displaystyle\int_{\bom} f\,\sigma\,d\mu:\: \int_{\bom}\sigma\,d\mu\,=\,1\,\}\\
 \label{upper-dual}
 \tau^*_t&=&\displaystyle\sup_{\lambda}\:\{\, \lambda:\: \M_t(f\cdot\bmu)\,\succeq\,\lambda\,\M_t(\bmu)\}
 \end{eqnarray}
 It is straightforward to see that $\tau_t\geq f^*$ for all $t$. Indeed let
 $\sigma\in\Sigma[\x]_t$ be a feasible solution of \eqref{upper-primal}. Then as $f\geq f^*$ for all $\x\in\bom$,
 \[\int_{\bom}f\,\sigma\,d\mu\,\geq\,f^*\,\int_{\bom}\sigma\,d\mu\,=\,f^*\,.\]
 Moreover, $\tau^*_t\leq\tau_t$ for every $t$ because from the definition of the localizing and moment matrices
 associated with $\bmu$ and $f$,
 \[\M_t(f\cdot\bmu)\,\succeq\,\lambda\,\M_t(\bmu)
 \Rightarrow \int_{\bom}f\,\sigma\,d\mu\geq\lambda\,\int_{\bom}\sigma\,d\mu\,,\quad\forall\sigma\in\Sigma[\x]_t\,,\]
 which in turn implies  $\lambda\leq\int_{\bom}f\,\sigma\,d\mu$ for all $\sigma$ feasible in \eqref{upper-primal},
 and therefore $\lambda\leq\tau_t$.

\begin{theorem}[\cite{siopt}]
Let $S\subset\R^n$ be compact with nonempty interior and $\tau_t$ and $\tau^*_t$ be as in  \eqref{upper-primal}
and \eqref{upper-dual} respectively. Then $\tau_t=\tau^*_t$ for every $t$ and $\tau_t\downarrow f^*$ as $t$ increases.
Moreover \eqref{upper-primal} (resp. \eqref{upper-dual})  has an optimal solution $\sigma^*\in\Sigma[\x]_t$
(resp. $\lambda^*$) and $\lambda^*$ is the smallest generalized eigenvalue of the pair of matrices 
$(\M_t(f\cdot\bmu),\M_t(\bmu))$ with associated eigenvector $\sigma^*$.
\end{theorem}

  The proof of the convergence $\tau_t\downarrow f^*$ as $t$ increases, is based on Theorem \ref{th:lasserre-pos}.
 The dual problem \eqref{upper-dual} has a single variable $\lambda$ and is a generalized eigenvalue problem associated with the pair of matrices $(\M_t(f\cdot\bmu),\M_t(\bmu))$. Therefore $\tau_t$ can be 
 computed by standard linear algebra routine with no optimization. See e.g. the discussion in \cite[Section 4]{siopt}. However the size of the involved matrices makes this technique quite difficult even for modest size problems. Nevertheless and fortunately, there is a variant \cite{univariate} that reduces to computing generalized eigenvalues of related \emph{univariate} Hankel moment matrices by using the pushforward (univariate) measure $\#\mu$
 (on the real line) of $\mu$ by $f$. That is,
 $\#\mu(B)=\mu(f^{-1}(B))$ for all $B\in\mathcal{B}(\R)$,
 and therefore
  \[f^*\,=\,\inf\,\{\,z\::\: z\,\in\,f(\bom)\,\}\,=\,\inf\,\{\,z\::\: z\,\in\,\mathrm{supp}(\#\mu)\,\}\,.\]
 Then letting $\H_t(\#\mu)$ (resp. $\H_t(z\cdot\#\mu)$) be the (univariate) Hankel moment matrix associated with $\#\mu$ (resp. $z\cdot\#\mu$), the sequence of scalars $(\delta_t)_{t\in\N}$ defined by
 \begin{equation}
 \label{univariate}
 \delta_t\,:=\,\displaystyle\sup_{\lambda}\{\,\lambda:\: \H_t(z\cdot \#\mu)\succeq\,\lambda\,\H_t(\#\mu)\,\}\,,\quad t\in\N\,,\end{equation}
 provides a monotone non-increasing sequence of upper bounds $(\delta_t)_{t\in\N}$ that converges to 
 $f^*$.  For more details, the interested reader is referred to \cite{univariate, slot-2}.
 
  When comparing  \eqref{univariate} with \eqref{upper-dual}, the gain in the computational burden 
 is striking. Indeed in \eqref{univariate} one has to compute generalized eigenvalues of 
 Hankel matrices of size $t+1$ instead of size ${n+t\choose t}$ in \eqref{upper-dual}. 
 Recent works in \cite{deKlerk-1,slot-1,slot-2,slot-3} have proven nice rates for the 
 convergence $\delta_t\downarrow f^*$  and $\tau_t\downarrow f^*$, 
 with an appropriate choice of the reference measure $\mu$ on 
 specific sets $\bom$ (e.g., sphere, box, simplex, etc.).
 Interestingly, the analysis makes use of sophisticated results about zeros of 
 orthogonal polynomials, and a clever perturbation of the Christoffel-Darboux kernel.
 
\section{The Christoffel-Darboux kernel and Christoffel functions}
In this section we briefly review basic properties of the Christoffel-Darboux (CD) kernel and Christoffel functions.
For more details on these classical tools,  the interested reader is referred to e.g. \cite{book, adv-comp} and the many references therein.

\subsection{Christoffel-Darboux kernel}
Let $\bom\subset\R^n$ be compact  with nonempty interior and let $\mu\in\mathscr{M}(\bom)_+$ 
be such that $\M_t(\mu)\succ0$ for all $t\in\N$.
 Let $(P_{\balpha})_{\balpha\in\N^n}\subset\R[\x]$ be a family of polynomials that are orthonormal with respect to $\mu$, 
 and view $\R[\x]_t$ as a finite-dimensional vector subspace 
 of the Hilbert space $L^2(\bom,\mu)$. Then 
  the kernel
   \begin{equation}
 \label{def-cd-kernel}
  (\x,\z)\mapsto K^\mu_t(\x,\z)\,:=\,\sum_{\balpha\in\N^n_t}P_{\balpha}(\x)\,P_{\balpha}(\z)\,,\quad \forall \x,\z\in\R^n\,,\:t\in\N\,,
  \end{equation}
  \noindent
 is called the Christoffel-Darboux (CD) kernel associated with $\mu$. It has an important property,
 namely it \emph{reproduces} $\R[\x]_t$. Indeed, for every $p\,\in\,\R[\x]_t$,
\begin{equation}
 \label{eq:reproducing}
 p(\x)\,=\,\int_{\bom}K^\mu_t(\x,\z)\,p(\z)\,d\mu(\z)\,,\quad\forall \x\in\R^n\,,
\end{equation}
 and for this reason, $(\R[\x]_t,K^\mu_t)$  is called a Reproducing Kernel Hilbert Space (RKHS). Then every $f\in L^2(\bom,\mu)$ can be approximated by a sequence of polynomials $(\hat{f}_t)_{t\in\N}$, where $\hat{f}_t\in\R[\x]_t$ for every $t\in\N$, and
 \[\x\mapsto \hat{f}_t(\x)\,:=\,\int_{\bom}f(\z)\,K^\mu_t(\x,\z)\,d\mu(\z)\,=\,\sum_{\balpha\in\N^n_t}
 \left(\int_{\bom}f(\z)\,P_{\balpha}(\z)\,d\mu(\z)\right)\,P_{\balpha}(\x)\,,\]
 so that $\Vert f-\hat{f}_t\Vert_{L^2(\bom,\mu)}\to 0$ as $t$ increases; see e.g. \cite[Section 2, p. 13]{book}.

 \subparagraph{Interpreting the reproducing property}
 Given $\y\in\R^n$ fixed, let $p\in\R[\x]_t$ be the polynomial defined by
 \begin{equation}
 \label{p-reproduce}
 \x\mapsto p(\x)\,:=\,K^\mu_t(\y,\x)\,,\quad\forall \x\in\R^n\,.
 \end{equation}
 Then by the reproducing property \eqref{eq:reproducing}, observe that
 \[\int_{\bom} \x^{\balpha}\,p(\x)\,d\mu(\x)\,=\,\y^{\balpha}\,=\,\int \x^{\balpha}\,\delta_{\{\y\}}(d\x)\,,\quad\forall\balpha\in\N^n_t\,,\]
 that is, viewing $p$ as a \emph{signed} density w.r.t. $\mu$, the signed measure $pd\mu$ on $\bom$,
 mimics the Dirac measure at $\y$,
 as long as only moments of order at most $t$ are concerned. This is illustrated in Figure  \ref{fig:kernel} where
 $\bom=[-1,1]$ and $d\mu=1_{[-1,1]}(x)dx$, $t$ varies between $1$ and $10$, and $y=0,1/2,1$.
 
 \begin{figure}[h]
\includegraphics[scale=.65]{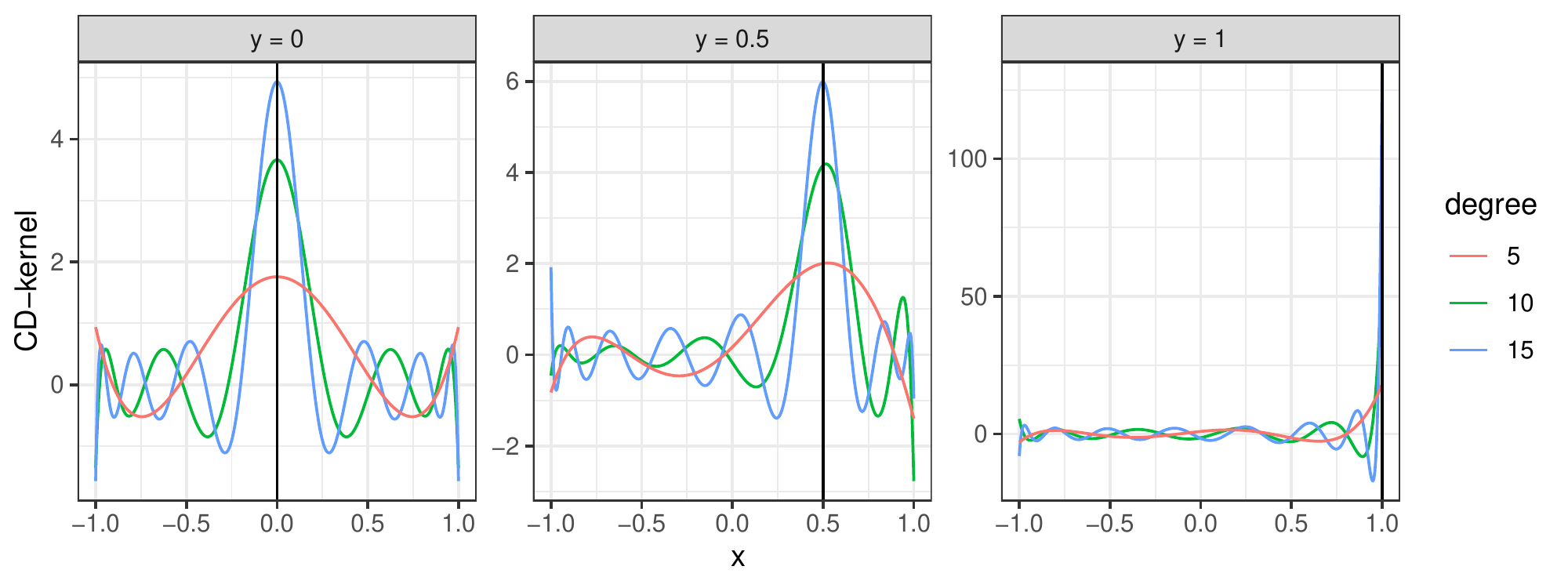}
\caption{The signed measure $K^\mu_t(y,x)dx$ (with different values of $t$) mimics the Dirac measure at 
$y=0$ (left) $y=0.5$ (middle) and $y=1$ (right); reprinted  from \cite[p. 45]{book} with permission.  \copyright Cambridge University Press}
\label{fig:kernel}       
\end{figure}
 
 \subsection{Christoffel function}
 With t$\in\N$, the function $\Lambda^{\mu}_t:\R^n\to\R_+$  associated with $\mu$, and defined by
\begin{equation}
\label{def-christo-1}
\x\mapsto \Lambda^\mu_t(\x)\,:=\,K^{\mu}_t(\x,\x)^{-1}\,=\,\left[\sum_{\balpha\in\N^n_t} P_{\balpha}(\x)^2\right]^{-1}\,,\quad\forall \x\in\R^n\,,\end{equation}
is called the (degree-$t$) Christoffel function (CF), and recalling that $\M_t(\mu)$ is nonsingular, it also turns out that 
\begin{equation}
\label{def-christo-11}
\Lambda^\mu_t(\x)\,=\,\left[\,\v_t(\x)^T\,\M_t(\mu)^{-1}\,\v_t(\x)\,\right]^{-1}\,,\quad\forall \x\in\R^n\,.
\end{equation}
The CF also has an equivalent and variational definition, namely:
\begin{eqnarray}
\label{def-christo-2}
\Lambda^\mu_t(\x)&=&\inf_{p\in\R[\x]_t}\{\,\int_{\bom} p^2\,d\mu\::\: p(\x)\,=\,1\,\}\, ,\quad\forall \x\in\R^n\\
\label{def-christo-3}
&=&\inf_{\p\in\R^{s(t)}}\{\,\langle\p,\M_t(\mu)\,\p\rangle\::\: \langle\p,\v_t(\x)\rangle\,=\,1\,\}\, ,\quad\forall \x\in\R^n\,.
\end{eqnarray}
In \eqref{def-christo-3} the reader can easily recognize a \emph{convex quadratic} programing problem
which can be solved efficiently even for large dimensions. However solving \eqref{def-christo-3} 
only provides the numerical value of 
$\Lambda^\mu_t$ at $\x\in\R^n$, whereas in \eqref{def-christo-11} one obtains the coefficients of the polynomial
$(\Lambda^\mu_t)^{-1}$ (but at the price of inverting $\M_t(\mu)$).

The reader will also notice that from its definitions \eqref{def-christo-1} or \eqref{def-christo-11}, the CF
depends only on the finite sequence $\bmu_{2t}=(\mu_{\balpha})_{\balpha\in\N^n_{2t}}$ of  moments of $\mu$, up to degree $2t$, and not on $\mu$ itself. Indeed there are potentially
many measures on $\bom$  with same moments up to degree $2t$, and therefore indexing $\Lambda^\mu_t$
with $\mu$ is not totally correct; therefore a more correct labelling would be $\Lambda^{\bmu_{2t}}_t$. One reason for this labelling is that in theory of approximation, one is usually given a measure $\mu$ on a compact set $\bom$ and 
one is interested in the sequence $(\Lambda^\mu_t)_{t\in\N}$ and its asymptotic properties.
 \begin{remark}
 In fact, one may also define 
 the CD-kernel $K_t^{\phi}$ and the Christoffel function (CF) $\Lambda^{\phi}_t$ associated with a 
 Riesz linear functional $\phi\in\R[\x]^*$ whose associated sequence $\bphi$ is such that $\M_t(\bphi)\succ0$,  no matter if $\phi$ is
 a measure  on $\bom$ or not. Indeed for fixed $t$,
 and letting $(P_{\balpha})_{\balpha\in\N^n_t}$ be orthonormal w.r.t. $\phi$, the polynomial
 \[(\x,\y)\mapsto K^{\bphi}_t(\x,\y)\,:=\,\sum_{\balpha\in\N^n_t}P_{\balpha}(\x)\,P_{\balpha}(\y)\,,\quad \forall \x,\y\in\R^n\,,\]
 is well-defined, and all definitions \eqref{def-cd-kernel}-\eqref{def-christo-3} are still valid.
 But again, historically  the CD-kernel was defined w.r.t. a given measure $\mu$ on $\bom$ .
 Finally, one may use interchangeably the notations $K^\phi_t$ (resp. $\Lambda^\phi_t$) or 
 $K^{\bphi}_t$ (resp. $\Lambda^{\bphi}_t$), or $K^{\bphi_{2t}}_t$ (resp. $\Lambda^{\bphi_{2t}}_t$) as in all cases, the resulting mathematical objet depends only 
 on the finite  moment sequence $\bphi_{2t}=(\phi_{\balpha})_{\balpha\in\N^n_{2t}}$ of $\bphi$.
 \end{remark}
 \subsection{Some distinguishing properties of the CF}
The CF $\Lambda^{\phi}_t$ associated with a Borel measure $\phi$ on 
a compact $\bom\subset\R^n$,  has an interesting and distinguishing feature.
As $t$ increases, $\Lambda^{\phi}_t(\x)\downarrow 0$ exponentially fast for every 
$\x\not\in\bom$ whereas its decrease is at most polynomial in $t$ whenever $\x\in\bom$; see e.g.
\cite[Section 4.3, p. 50--51]{book}.
 In other words, $\Lambda_t^\phi$ identifies the support of $\phi$ when $t$ is sufficiently large. In addition, at least in dimension $n=2$ or $n=3$, one may visualize this property even for small $t$, as the resulting superlevel sets 
$\bom_\gamma:=\{\,\x: \Lambda^\phi_t(\x)\geq \gamma\,\}$, $\gamma\in\R$, capture the geometric shape of $\bom$ 
quite accurately; For instance in Figure
\ref{fig:cd-shape} are displayed
several level sets
$\bom_\gamma$ associated with the empirical measure $\phi_N$ supported on 
a cloud of $N$ points that approximates 
the geometric shape obtained with the letters ``C" and ``D" of Christoffel and Darboux. In \cite{neurips},
the interested reader can find many other examples 
of $2D$-clouds with non-trivial geometric shapes which are captured quite well with levels set 
$\bom_{\gamma}$ associated with $\phi_N$, even for relatively low degree $t$.
 \begin{figure}[ht]
\includegraphics[scale=.40]{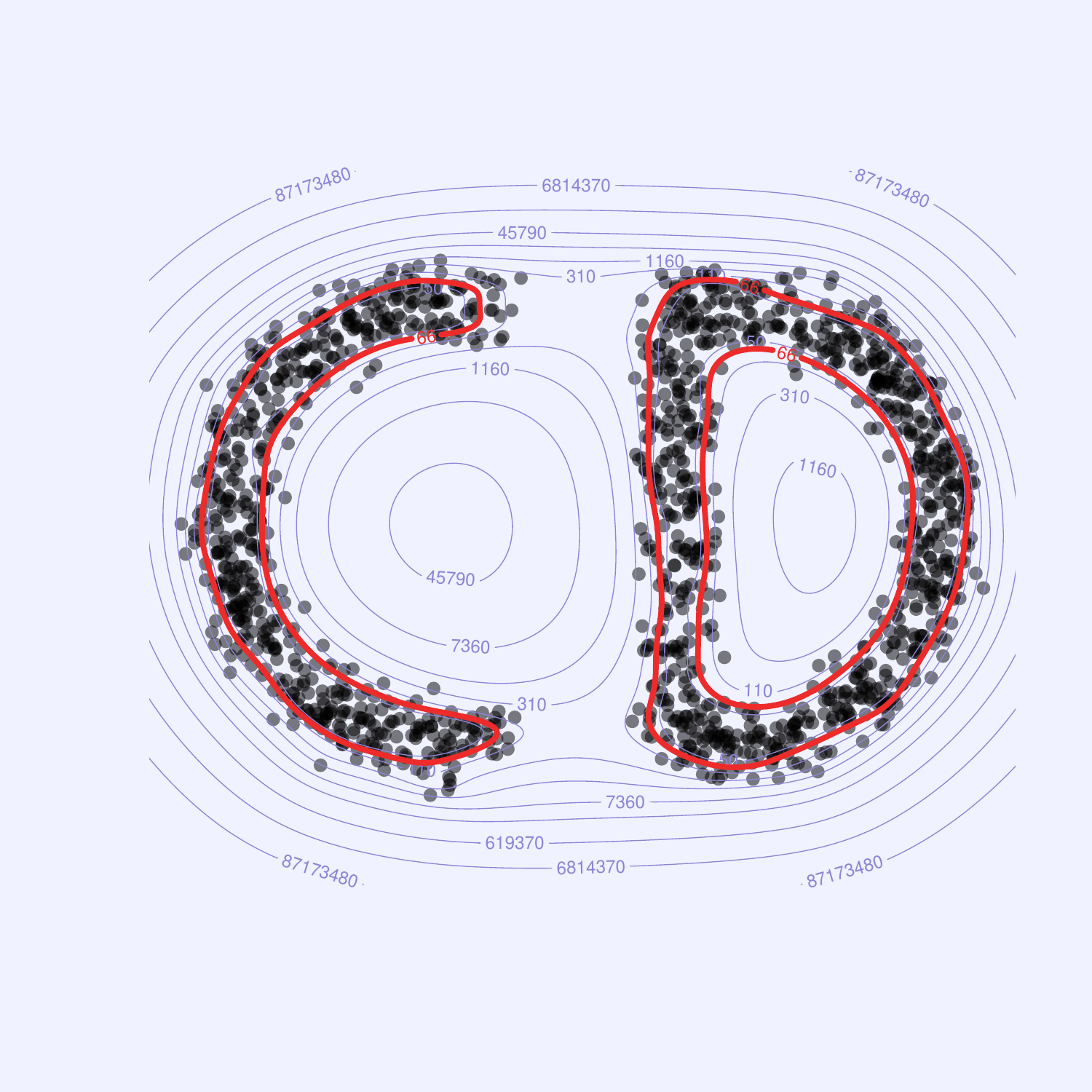}
\caption{Level sets $\bom_\gamma$ associated with $\Lambda^{\phi_N}_{10}$ for various values of $\gamma$; the red level set is obtained with $\gamma:=s(10)$; reprinted  from cover of \cite[p. 45]{book} with permission.  \copyright  Cambridge University Press}
\label{fig:cd-shape}       
\end{figure}

Another nice feature of the CF is its ability to approximate densities. Indeed let $\phi$ and $\mu$ be 
finite Borel measures on a compact set $\bom$, and let
$\mu$ be such that uniformly on compact subsets of $\mathrm{int}(\bom)$,
$\lim_{t\to\infty}s(t)\,\Lambda^\mu_t=h_\mu$, where $h_\mu$ is continuous and positive on $\mathrm{int}(\bom)$ (and recall that $s(t)$ is the dimension of $\R[\x]_t$). In addition suppose that $\phi$ has continuous and positive density 
$f_\phi$  w.r.t. $\mu$. Then uniformly on compact subsets of $\mathrm{int}(\bom)$
\[\lim_{t\to\infty} s(t)\,\Lambda^\phi_t\,=\,f_\phi\cdot h_\mu\,\]
(see e.g. \cite[Theorem 4.4.1]{book}). So if the function $h_\mu$ is already known then one can approximate the density $f_\phi$, uniformly on compact subsets of $\mathrm{int}(\bom)$.

Finally, another distinguishing property of the CF is its link with the so-called \emph{equilibrium measure} of 
the compact set $\bom$. The latter 
is a measure on $\bom$ (let us denote it by $\lambda_{\bom}$) 
which minimizes some Riesz energy functional
(invoking pluripotential theory and viewing $\R^n$ as a subset of $\C^n$). For a detailed treatment see
e.g. \cite{bedford}. The measure $\lambda_{\bom}$ 
is known only for sets with specific geometry (e.g., an interval of the real line, the 
simplex, the unit sphere, the unit euclidean unit box). However under some condition\footnote{The set 
$\bom$ is assumed to be regular and $(\bom,\phi)$ possesses the Bernstein-Markov property; see e.g. \cite[Section 4.4]{book}}, as $t$ increases, the Borel measure $\nu_t$
on $\bom$ with density $1/s(t)\Lambda^\phi_t$ w.r.t. $\phi$, converges to 
$\lambda_{\bom}$  in the weak-$\star$ topology of $\mathscr{M}(\bom)$ 
(the Banach space of finite signed Borel measures on $\bom$ equipped with the total variation norm). That is:
\[\lim_{t\to\infty} \int_{\bom} h\,d\nu_t\,:=\,\lim_{t\to\infty} \int_{\bom}\frac{h}{s(t)\Lambda^\phi_t}\,d\phi
\,=\,\int_{\bom} h\,d\lambda_{\bom}\,,\quad\forall h\in\mathscr{C}(\bom)\,\]
where $\mathscr{C}(\bom)$ is the space of continuous functions on $\bom$; 
(see e.g. \cite[Theorem 4.4.4]{book}). In particular, the moments $\nu_{t,\balpha}$, $\balpha\in\N^n_t$, converge to the moments of $\lambda_{\bom}$.

\section{CF, Optimization, and SOS-Certificates of Positivity}

\subsection{The CF to compare the hierarchies of upper and lower bounds}
Recall the polynomial optimization problem $\P$ in \eqref{def-pb-P} with ${\bom}\subset\R^n$ as in \eqref{set-S}. Let $\mu$ be a finite Borel (reference) measure whose support is exactly ${\bom}$ and with an associated sequence of orthonormal polynomials $(P_\alpha)_{\alpha\in\N^n}$.
Next, with $\bphi\in\R^{s(2t)}$ and from the reproducing property  \eqref{eq:reproducing}, observe that
\begin{eqnarray*}
 \bphi(f)&=&\bphi\left(\int_{\bom} \sum_{\balpha\in\N^n_{2t}}P_{\balpha}(\x)P_{\balpha}(\y)\,f(\y)\,d\mu(\y)\right)\\
 &=& \sum_{\balpha\in\N^n_{2t}}\bphi(P_{\balpha})\,\int_{\bom}P_{\balpha}(\y)\,f(\y)\,d\mu(\y)\\
 &=& \int_{\bom}f(\y)\,\sum_{\balpha\in\N^n_{2t}}\bphi(P_{\balpha})\, P_{\balpha}(\y)\,d\mu(\y)\,=\,\int_{\bom}f(\y)\,\sigma_{\bphi}(\y)\,d\mu(\y)
\end{eqnarray*}
where the degree-$2t$ polynomial $\y\mapsto \sigma_{\bphi}(\y):=\sum_{\balpha\in\N^n_{2t}}\bphi(P_{\balpha})\,P_{\balpha}(\y)$, is a \emph{signed density} w.r.t. $\mu$. 

Therefore in the semidefinite relaxations \eqref{relax-primal} of lower bounds on $f^*$, one searches for a linear functional 
$\phi\in\R[\x]_{2t}^*$ which satisfies
\[\phi(1)\,=\,1\,;\quad \M_t(g\cdot\bphi)\,\succeq\,0\,,\quad\forall g\in G\,,\]
and which minimizes $\phi(f)=\int_{\bom}f\,\sigma_{\bphi}\,d\mu$, where
$\sigma_{\bphi}$ is a degree-$2t$ 
polynomial signed density w.r.t. $\mu$, with coefficients $(\sigma_{\bphi,\balpha}:=\bphi(P_{\balpha}))_{\balpha\in\N^n_{2t}}$.

The reason why the semidefinite relaxations \eqref{relax-primal} can be exact (i.e., $\rho_t=f^*$ for some $t$), is that the signed probability measure $d\nu_t=\sigma_{\bphi}d\mu$ can mimic the Dirac measure at a global minimizer $\bxi\in\bom$ and so $\phi(f)=f(\bxi)$; see Figure \ref{fig:kernel}. 

This is in contrast to the hierarchy of semidefinite relaxations \eqref{upper-primal} of upper bounds where one searches also for a polynomial density $\sigma\,d\mu$ w.r.t. $\mu$, but as this density $\sigma$ is an SOS (hence positive),  it cannot be a  Dirac measure, and therefore the resulting convergence $\tau_t\downarrow f^*$ is necessarily asymptotic and \emph{not} finite. For more details on a comparison between the Moment-SOS hierarchies of upper and lower bounds, the interest reader is referred to \cite{comparison}.

\subsection{The CF and positive polynomials}
Of course, from its definition \eqref{def-christo-11} the reciprocal $(\Lambda^\mu_t)^{-1}$ of the CF 
is an SOS of degree $2t$. But we next reveal an even more interesting link with SOS polynomials.
Observe that the $Q_t(G)$ in \eqref{set-Q_t(G)} is a convex cone and its dual reads

\begin{equation}
\label{cone-dual}
Q_t(G)^*\,=\,\{\,\bphi=(\phi_{\balpha})_{\balpha\in\N^n_{2t}}: \M_{t-t_g}(g\cdot\bphi)\,\succeq\,0\,,\quad \forall\,g\in G\,\}\,.
\end{equation}
\newpage
\noindent
{\bf A duality result of Nesterov.} 

 \begin{lemma}
\label{lem:nesterov}
If $p\in\mathrm{int}(Q_t(G))$ then there exists $\bphi\in\mathrm{int}(Q_t(G)^*)$ such that
\begin{eqnarray}
 \label{nesterov-1}
 p&=&\sum_{g\in G} g(\x)\,\v_{t-t_g}(\x)^T\M_{t-t_g}(g\cdot\bphi)^{-1}\v_{t-t_g}(\x)\,,\quad\forall \x\in\R^n\\
 \label{nesterov-2}
 &=&\sum_{g\in G} g\cdot (\Lambda^{g\cdot\bphi}_t)^{-1}\,.
\end{eqnarray}
In particular, for
every SOS $p\in\mathrm{int}(\Sigma_t[\x])$, $1/p$ is the CF of some linear functional 
$\phi\in\Sigma[\x]^*_{2t}$, i.e., $1/p=\Lambda^{\bphi}_t$ for some
$\bphi\in\R^{s(2t)}$ such that $\M_t(\bphi)\succ0$. In addition, in the univariate case, $\bphi$ has a representing measure on $\R$.
 \end{lemma}

Equation \eqref{nesterov-1} is from \cite{nesterov} while its interpretation
\eqref{nesterov-2} is from \cite[Lemma 4]{cras-1}. 
Observe that \eqref{nesterov-2} provides  a distinguished representation of $p\in\mathrm{int}(Q_t(G))$,
and in view of its specific form, we propose to name \eqref{nesterov-2} the 
\emph{Christoffel representation} of $p\in\mathrm{int}(Q_t(G))$, that is:

\begin{equation}
 \label{eq:Chris-rep-Q(G)}
 \mathrm{int}(Q_t(G))\,=\,\{\,\sum_{g\in G} g\cdot (\Lambda^{g\cdot\bphi}_t)^{-1}\::\: \bphi\,\in\,\mathrm{int}(Q_t(G)^*)\,\}\,.
\end{equation}

Of course, 
an intriguing question is: \emph{What is the link between 
$\bphi\in \mathrm{int}(Q_t(G)^*)$ in \eqref{nesterov-2} and the polynomial
$p\in\mathrm{int}(Q_t(G))$?} A partial answer is provided in Section \ref{CF-equi}.\\

\noindent
{\bf A numerical procedure to obtain the Christoffel representation.} 
Consider the following optimization problems:
\begin{equation}
\label{eq:primal}
\begin{array}{rl}
\mathbf{P}:\quad \displaystyle\inf_{\bphi\in\R^{s(2t)}}&\{\,-\displaystyle\sum_{g\in G}\log\mathrm{det}(\M_{t-t_g}(g\cdot\bphi)):\:
\phi(p)=\displaystyle\sum_{g\in G}s(t-t_g)\\
&\M_{t-t_g}(g\cdot\bphi)\,\succeq\,0\,,\quad \forall g\in G\,\}\,.
\end{array}
\end{equation}
\begin{equation}
\label{eq:dual}
\begin{array}{rl}
\mathbf{P}^*:\quad \displaystyle\sup_{\Q_g\succeq0}&\{\,\displaystyle\sum_{g\in G}\log\mathrm{det}(\Q_g):\\
&p(\x)\,=\,\displaystyle\sum_{g\in G} g(\x)\,\v_{t-t_g}(\x)^T\Q_g\v_{t-t_g}(\x)\,,\quad\forall \x\in\R^n\,\}\,.
\end{array}
\end{equation}
Both $\P$ and $\P^*$ are convex optimization problems that can be solved by off-the-shelf software packages like e.g. CVX \cite{cvx}.

\begin{theorem}
\label{th:primal-dual}
Let $p\in\mathrm{int}(Q_t(G))$.  Then $\P^*$ is a dual of $\P$, that is, for every feasible solution 
$\bphi\in\R^{s(2t)}$ of \eqref{eq:primal} and $(\Q_g)_{g\in G}$ of \eqref{eq:dual},
\begin{equation}
\label{th:primal-dual-weak}
\displaystyle\sum_{g\in G}\log\mathrm{det}(\Q_g)\,\leq\,
-\displaystyle\sum_{g\in G}\log\mathrm{det}(\M_{t-t_g}(g\cdot\bphi))\,.\end{equation}
Moreover,
both $\P$ and $\P^*$ have a unique optimal solution $\bphi^*$ and $(\Q^*_g)_{g\in G}$ respectively, 
which satisfy
\begin{equation}
\label{th:primal-dual-1}
\Q^*_g\,=\,\M_{t-t_g}(g\cdot\bphi^*)^{-1}\,,\quad\forall g\in G\,,
\end{equation}
and which yields equality in \eqref{th:primal-dual-weak}.
 \end{theorem}

The proof which uses Lemma \ref{lem-logdet} in Appendix, mimics that of \cite[Theorem 3]{cras-1} (where $p$ was a constant polynomial) and is omitted. 

\subsection{A disintegration of the CF}
We next see how the above duality result, i.e., the Christoffel representation \eqref{nesterov-2} of $\mathrm{int}(Q_t(G))$, can be used to in turn infer a \emph{disintegration property} 
of the CF. So let $\Lambda^\mu_t(\x,y)$ be the CF of a Borel probability measure 
$\mu$ on $\bom\times Y$, where $\bom\subset\R^n$ and $Y\subset\R$ are compact.
It is well-known that $\mu$ disintegrates into its marginal probability $\phi$ on $\bom$, and 
a conditional measure $\hat{\mu}(dy\vert \x)$ on $Y$, given $\x\in\bom$, that is,
\[\mu(A\times B)\,=\,\int_{\bom\cap A}\hat{\mu}(B\vert\x)\,\phi(d\x)\,,\quad\forall A\in\mathcal{B}(\bom)\,,\,B\in\mathcal{B}(Y)\,.\]

\begin{theorem}[\cite{cras-1}]
\label{th-main}
Let $\bom\subset\R^n$ (resp. $Y\subset\R$) be compact with nonempty interior, and let
$\mu$ be a Borel probability measure on $\bom\times Y$, with marginal $\phi$ on $\bom$.
  Then for every $t\in\N$, and $\x\in\bom$, there exists a probability measure $\nu_{\x,t}$ on $\R$ such that
  \begin{equation}
  \label{th-main-1}
 \Lambda^\mu_t(\x,y)\,=\,\Lambda^\phi_t(\x)\cdot\Lambda^{\nu_{\x,t}}_t(y)\,,\quad\forall \x\in\R^n\,,\:y\in\R\,.
\end{equation}
 \end{theorem}
 
The proof in \cite[Theorem 5]{cras-1} heavily relies on the duality result of Lemma \ref{lem:nesterov}.
In particular every degree-$2t$ univariate SOS $p$ in the interior of $\Sigma[y]_t$ is the reciprocal of the Christoffel function of 
some Borel measure on $\R$. 

\subsection{Positive polynomials and equilibrium measure}
\label{CF-equi}
This section is motivated by the following observation. Let $(T_n)_{n\in\N}$ (resp. $(U_n)_{n\in\N}$) be the family of Chebyshev polynomials of the first kind (resp. second kind). They are orthogonal  
w.r.t. measures $(1-x^2)^{-1/2}dx$ and $(1-x^2)^{1/2}dx$ 
on $[-1,1]$, respectively. (The Chebyshev measure $(1-x^2)^{-1/2}dx/\pi$ 
is the \emph{equilibrium measure} of the interval $[-1,1]$.) They also satisfy the identity
\[T_n(x)+(1-x^2)\,U_n(x)\,=\,1\,,\quad\forall x\in \R\,,\quad n=1,\ldots\]
Equivalently, it is said that the triple $(T_n,(1-x^2),U_n)$ is a solution to (polynomial) Pell's equation for every $n\geq1$. For more details on polynomial Pell's equation (originally Pell's equation is a topic in algebraic number theory), the interested reader is referred to \cite{pell-2,pell-1}.
Next, letting $x\mapsto g(x):=(1-x^2)$, and after normalization to pass to orthonormal polynomials, in summing up 
one obtains
\begin{equation}
 \label{pell}
 \Lambda^\phi_t(x)^{-1}+(1-x^2)\,\Lambda^{g\cdot\phi}(x)^{-1}\,=\,2t+1\,,\quad \forall x\in\R\,,\: \forall t=0,1,\ldots
\end{equation}
Now, invoking Lemma \ref{lem:nesterov}, observe that \eqref{pell} also states that $p\in \mathrm{int}(Q_t(G))$ where $G=\{g\}$ and
$p$ is  the constant polynomial $x\mapsto p(x)=2t+1$. In addition, it also means that if one solves $\P$ in \eqref{eq:primal} with $p=s(t)+s(t-t_g)=2t+1$ (recall that $t_g=1$), then its unique optimal solution $\bphi$
is just the vector of moments (up to degree $2t$) of the equilibrium measure $\phi=(1-x^2)^{-1/2}dx/\pi$ of the interval $\bom=[-1,1]$. 
So in Lemma \ref{lem:nesterov} the linear functional $\phi_p$ associated with
the constant polynomial $p=2t+1$ is simply the equilibrium measure of $\bom$ (denote it $\lambda_{\bom}$). 

The notion of equilibrium measure associated to a given set
originates from logarithmic potential theory (working in
$\mathbb{C}$ in the univariate case to minimize some energy functional) and some generalizations have been obtained in the multivariate case 
via pluripotential theory in $\mathbb{C}^n$. 
In particular if $\bom\subset\R^n\subset\mathbb{C}^n$ is compact then 
the equilibrium measure $\lambda_{\bom}$ is equivalent to Lebesgue measure on compact subsets of $\mathrm{int}(\bom)$.
See e.g. 
Bedford and Taylor \cite[Theorem 1.1]{bedford} and \cite[Theorem 1.2]{bedford}. 

\subparagraph{The Bernstein-Markov property} A measure with compact support $\bom$ satisfies the Bernstein-Markov property if there exists a sequence of positive numbers $(M_t)_{t\in\N}$ such that for all $t\in\N$ and all $p\in\R[\x]_t$,
\[\sup_{\x\in\bom}\vert p(\x)\vert\:(=\Vert p\Vert_{\bom})\:\leq\,M_t\cdot \Vert p\Vert_{L^2(\bom,\mu)}\,,\]
and $\lim_{t\to\infty}\log(M_t)/t=0$. 

So when it holds, the Bernstein-Markov property 
describes  how the sup-norm and the $L^2(\bom,\mu)$-norm of polynomials relate when the degree increases.

In \cite{cras-3} we have obtained the following result. Let $\x\mapsto \theta(\x):=1-\Vert\x\Vert^2$ and possibly after an appropriate scaling, 
let $\bom$ in \eqref{set-S} be such that $\theta\in Q_1(G)$ (so that
$\bom\subset [-1,1]^n$); see Remark \ref{rem:archimedean}.

\begin{theorem}[\cite{cras-3}]
\label{th0}
Let $\phi\in\R[\x]^*$ (with $\phi_0=1$) be such that 
$\M_t(g\cdot\bphi)\succ0$ for all $t\in\N$ and all $g\in G$, so that the Christoffel functions $\Lambda^{g\cdot\phi}_t$ are all well defined.
In addition, suppose that there exists $t_0\in\N$ such that
 \begin{equation}
 \label{th0-1}
 \sum_{g\in G_t}s(t-t_g)\,=\,\sum_{g\in G_t} g\cdot(\Lambda^{g\cdot\phi}_{t-t_g})^{-1}\,,\quad\forall t\geq t_0\,.
 \end{equation}
 \indent Then: (a)  for every $t\geq t_0$, the finite moment sequence
 $\bphi^*_t:=(\phi_{\balpha})_{\balpha\in\N^n_{2t}}$ is the unique optimal solution of \eqref{eq:primal} (with $p$ the constant polynomial $\x\mapsto \sum_{g\in G}s(t-t_g)$).
 
  \indent (b) $\phi$ is a Borel measure on $\bom$ and the 
 unique representing measure of $\bphi$. Moreover, if 
 $(\bom,g\cdot\phi)$ satisfies the Bernstein-Markov property for every $g\in G$, then 
 $\phi$ is the equilibrium measure $\lambda_{\bom}$ and  therefore the Christoffel polynomials $(\Lambda^{g\cdot\lambda_{\bom}}_{t})^{-1}_{g\in G_t}$ satisfy the generalized Pell's equations:
  \begin{equation}
 \label{th0-2}
 \sum_{g\in G_t}s(t-t_g)\,=\,\sum_{g\in G_t} g\cdot(\Lambda^{g\cdot\lambda_{\bom}}_{t-t_g})^{-1}\,,\quad\forall t\geq t_0\,.
 \end{equation}
 \end{theorem}

Importantly, the representation of $\bom$ in \eqref{set-S} 
depends on the chosen set $G$ of generators, which is not unique. Therefore if \eqref{th0-1} holds for some set $G$, it may not hold for another set $G$. 
The prototype of $\phi$ in Theorem \ref{th0} is the equilibrium measure of $\bom=[-1,1]$, i.e.,
the Chebyshev measure $(1-x^2)^{-1/2}dx/\pi$ on $[-1,1]$.
So Theorem \ref{th0} is a strong result which is likely to hold only for quite specific sets $\bom$ (and provided that a good set of generators is used). In \cite{cras-3} the author could prove that \eqref{th0-1} also holds for the equilibrium measure $\lambda_{\bom}$ of the $2D$-simplex, the $2D$-unit 
unit box, the $2D$-Euclidean unit ball, at least for $t=1,2,3$. 

However, if $\theta\in Q_1(G)$ then as proved in \cite{cras-3,ctp}, $1\in\mathrm{int}(Q_t(G))$ for all $t$, and therefore \eqref{eq:primal} has always a unique optimal solution $\bphi^*_t$. That is, for every $t\geq t_0$, \eqref{th0-1} hold for some $\phi_t\in\R[\x]_{2t}^*$ which \emph{depends on $t$} (whereas in \eqref{th0-1} one considers moments up to degree $2t$ of the \emph{same} $\phi$).
Moreover, every accumulation point $\bphi$ of the sequence $(\bphi^*_t)_{t\in\N}$ has a representing measure $\phi$ on $\bom$. An interesting issue to investigate is the nature of $\phi$, in particular its relationship with the equilibrium measure $\lambda_{\bom}$ of $\bom$.

Finally, for general compact sets $\bom$ with nonempty interior,
to $\lambda_{\bom}$ one may associate
the polynomial
\[p^*_t\,:=\,\frac{1}{\sum_{g\in G_t}s(t-t_g)}\sum_{g\in G_t}g\cdot (\Lambda_t^{g\cdot\lambda_{\bom}})^{-1}\,\]
which is well-defined because the matrices $\M_{t-t_g}(g\cdot\lambda_{\bom})$ are non singular.
In Theorem \ref{th0} 
one has considered cases where $p^*_t$ is exactly the constant (equal to $1$) polynomial
(like for the Chebyshev measure on $\bom=[-1,1]$). We now consider the measures
$(\mu_t:=p^*_t\lambda_{\bom})_{t\in\N}$,
with respective densities $p^*_t$ w.r.t. $\lambda_{\bom}$. Each $\mu_t$ is a probability measure on $\bom$ because
\begin{eqnarray*}
\int p^*_t\,d\lambda_{\bom}&=&
\frac{1}{\sum_{g\in G_t}s(t-t_g)}\sum_{g\in G_t}\int g\cdot (\Lambda_t^{g\cdot\lambda_{\bom}})^{-1}\,d\lambda_{\bom}\\
&=&\frac{1}{\sum_{g\in G_t}s(t-t_g)}\sum_{g\in G_t}\langle \M_{t-t_g}(g\cdot\lambda_{\bom}),\M_{t-t_g}(g\cdot\lambda_{\bom})^{-1}\rangle\\
&=&\frac{1}{\sum_{g\in G_t}s(t-t_g)}\sum_{g\in G_t}s(t-t_g)\,=\,1\,.
\end{eqnarray*}
Moreover, preceding as in the proof of Theorem \ref{th0} in \cite{cras-3}, it follows that 
\[\lim_{t\to\infty}\int \x^{\balpha}\,p^*_t\,d\lambda_{\bom}\,=\,\int \x^{\balpha}\,d\lambda_{\bom}\,,\quad\forall\balpha\in\N^n\,.\]
As $\bom$ is compact it implies that the sequence of probability measures $(\mu_t)_{t\in\N}\subset\mathscr{M}(\bom)_+$ 
converges to $\lambda_{\bom}$ for the weak-$\star$ topology of $\mathscr{M}(\bom)$. In other words (and in an informal language), the density 
$p^*_t$ of $\mu_t$ w.r.t. $\lambda_{\bom}$ behaves like the constant (equal to $1$) polynomial,
which can be viewed as a weaker version of \eqref{th0-2}.

\section{Conclusion}

SOS polynomials play a crucial role in the 
Moment-SOS hierarchies of upper and lower bounds through their use in 
certificates of positivity of real algebraic geometry. We have shown that
they are also related to the Christoffel function
in theory of approximation. Interestingly, the link is provided by interpreting a duality result in convex optimization
applied to a certain convex cone of polynomials and its dual cone of pseudo-moments. It also turns out that in this cone, the constant  polynomial  is strongly related
to the equilibrium measure of the semi-algebraic set associated with the convex cone. We hope that these interactions between different and seemingly disconnected fields will raise the curiosity of the optimization community and yield further developments.

\section*{Appendix}
\label{sec-appendix}
\begin{lemma}
\label{lem-logdet}
Let $\mathcal{S}^n$ be the space of real symmetric $n\times n$ matrices and let $\mathcal{S}^n_{++}\subset\mathcal{S}^n$ 
be the convex cone of real $n\times n$ positive definite matrices $\Q$ (denoted $\Q\succ0$). Then
\begin{equation}
\label{fenchel}
n+\log\mathrm{det}(\M)+\log\mathrm{det}(\Q)\,\leq\,\langle\M,\Q\rangle\,,\quad\forall 
\M\,,\Q\,,\in\,\mathcal{S}^n_{++}\,.
\end{equation}
with equality if and only if $\Q=\M^{-1}$.
\end{lemma}
\begin{proof}
  Consider the concave function
\[f:\: \mathcal{S}^n\to \R\cup\{-\infty\}\,\quad \Q\mapsto f(\Q)\,=\,\left\{\begin{array}{l}
\log\mathrm{det}(\Q)\:\mbox{ if $\Q\in\mathcal{S}^n_{++}$,}\\
-\infty\:\mbox{ otherwise,}\end{array}\right.\]
and let $f^*$ be its (concave analogue) of Legendre-Fenchel conjugate, i.e.,
\[\M\,\mapsto f^*(\M)\,:=\,\inf_{\Q\in\mathcal{S}^n}\langle \M,\Q\rangle-f(\Q)\,.\]
It turns out that
\[f^*(\M)\,=\,\left\{\begin{array}{l}n+\log\mathrm{det}(\M)\,(=\,n+f(\M))\mbox{ if $\M\,\in\,\mathcal{S}^n_{++}$ ,}\\
-\infty\mbox{ otherwise.}\end{array}\right.\]
Hence the concave analogue of Legendre-Fenchel inequality states that
\[f^*(\M)+f(\Q)\,\leq\,\langle \M,\Q\rangle\,,\quad\forall \M\,,\Q\,\in\,\mathcal{S}^n\,,\]
and yields  \eqref{fenchel}.
\end{proof}

\end{document}